\documentclass[11pt]{amsart}

\usepackage{amsmath}
\usepackage{amsfonts}
\usepackage{amssymb}
\usepackage{amsthm}
\usepackage{indentfirst}
\usepackage[T1]{fontenc}

\usepackage{hyperref}

\usepackage{tikz}
\usepackage{verbatim}
\usetikzlibrary{matrix}

\usepackage{enumerate}

\newtheorem{theorem}{Theorem}
\newtheorem{thm}[theorem]{Theorem}
\newtheorem{lem}[theorem]{Lemma}
\newtheorem{prop}[theorem]{Proposition}

\newtheorem{rmk}[theorem]{Remark}

\newcommand{\fm}{{\mathfrak m}}
\newcommand{\ba}{{\vec a}}

\newcommand{\Q}{{\mathbb{Q}}}

\newcommand{\Z}{{\mathbb{Z}}}

\newcommand{\cO}{{\mathcal{O}}}

\newcommand{\SL}{\operatorname{SL}}
\newcommand{\GL}{\operatorname{GL}}
\newcommand{\Tr}{\operatorname{Tr}}

\def\mm#1#2#3#4{\begin{pmatrix}#1 & #2 \\ #3 & #4\end{pmatrix}}
\begin{document}

\begin{abstract}
If $R$ is the ring of integers of a number field, then there exists a \emph{polynomial parametrization} of the set $\SL_2(R)$, i.e., an element
$A\in \SL_2(\Z[x_1,\ldots,x_n])$ such that every element of $\SL_2(R)$ is obtained by specializing $A$ via some homomorphism $\Z[x_1,\ldots,x_n]\to R$.
\end{abstract}

\title{Polynomial Parametrization for $\SL_2$ over Quadratic Number Rings}

\author{Michael Larsen}

\address{Department of Mathematics \\
          Indiana University \\
Bloomington, Indiana 47405, USA}

\email{\href{mailto:mjlarsen@indiana.edu}{\tt mjlarsen@indiana.edu}}

\author{Dong Quan Ngoc Nguyen}

\address{Department of Applied and Computational Mathematics and Statistics \\
         University of Notre Dame \\
         Notre Dame, Indiana 46556, USA }

\email{\href{mailto:dongquan.ngoc.nguyen@nd.edu}{\tt dongquan.ngoc.nguyen@nd.edu}}

\date{August 16, 2018}

\thanks{ML was partially supported by NSF Grant DMS-1702152. \\
DQNN was partially supported by DARPA grant N66001-17-1-4041.}

\maketitle

Let $R$ be a commutative ring.  We say a subset $S\subset \SL_N(R)$ is \emph{bounded} if there exists an element $A(x_1,\ldots,x_n)\in \SL_N(\Z[x_1,\ldots,x_n])$
such that
$$S\subseteq \{A(r_1,\ldots,r_n)\mid r_i\in R\},$$
and $A(x_1,\ldots,x_n)=I$ has a solution in $R^n$, where $I$ denotes the identity matrix. It is clear that if $S$ and $T$ are bounded subsets of $\SL_N(R)$, then every subset of $S$ is bounded, and likewise, 
$S\cup \{I\}$, $S^{-1}$ and $ST$ are bounded.  Thus, $S\cup T \subset (S\cup\{I\})(T\cup\{I\})$ is also bounded.
When $\SL_N(R)$ itself is bounded, we say it is \emph{polynomially parametrized}.

For $1\le i\neq j\le N$, the set of elementary matrices $\{e_{ij}^r\mid r\in R\}$, with entry $r$ in position $(i,j)$, is bounded.
Therefore the set of all elementary matrices (i.e., the union of these sets over pairs $(i,j)$) is again bounded, so for 
any fixed $k$, the set of products of $k$ elementary matrices is bounded.  
Carter and Keller \cite{CK} proved that if $N\ge 3$ and $R$ is the ring of integers in any number field, then every element of $\SL_N(R)$ can  be written as a 
product of $k$ elementary matrices, for $k$ depending on $N$ and $R$.  Thus $\SL_N(R)$ is  polynomially parametrized.
This leaves the question as to whether $\SL_2(R)$ is likewise always polynomially parametrized.

When $R^\times$ is infinite, this is known to have an affirmative answer.
Vaserstein \cite{V1}  proved that in this case $\SL_2(R)$ is generated by elementary matrices, and
Carter, Keller, and Paige \cite{M} proved that this implies that $\SL_2(R)$ is indeed polynomially parametrized (see also recent work by Morgan, Rapinchuk, and Sury \cite{MRS} for another proof of this fact.)  Vaserstein also proved \cite{V2}  that $\SL_2(\Z)$ is polynomially parametrized.
This leaves the case of rings of integers in imaginary quadratic fields.  The point of this note is to show
that the methods of Carter, Keller in \cite{CK} and Vaserstein in \cite{V2} extend to cover
this case as well.

We say $Z\subseteq \SL_2(R)^2$ is \emph{bounded}  if there exist bounded sets $S$ and $T$ in $\SL_2(R)$ such that for all 
pairs $(M,N)\in Z$ there exist $X\in S$ and $Y\in T$ such that 
$$N = X M Y.$$
In particular, $\{(I,X)\mid X\in S\}$ is bounded if and only if $S$ is bounded in $\SL_2(R)$.

\begin{lem}
\label{equiv}
We have the following boundedness statements for $\SL_2(R)^2$:
\begin{enumerate}
\item The set $\{(M,M)\mid M \in \SL_2(R)\}$ is bounded.  
\item If $Z\subseteq \SL_2(R)^2$ is bounded, then $\{(M,N)\mid (N,M)\in Z\}$ is bounded.
\item If $Z,W\subseteq  \SL_2(R)^2$ are bounded, then the set of pairs $(M,P)\in \SL_2(R)^2$ such that there exists $N$ with $(M,N)\in Z$ and $(N,P)\in W$ is bounded.
\item The set of pairs $\{(M,XMY)\mid M\in \SL_2(R),\,X,Y\in \SL_2(\Z)\}$ is bounded.
\item The set $\{(M^{-1},M^T)\mid M\in \SL_2(R)\}$ is bounded.
\end{enumerate}
\end{lem}

\begin{proof}
Part (1) is trivial.  Part (2) follows from the fact that if $S\subseteq \SL_2(R)$ is bounded, then $S^{-1}$ is bounded.
Part (3) follows from the fact that if $S,T\subseteq \SL_2(R)$ are bounded, then $ST$ is bounded.
Part (4) follows from the boundedness of $\SL_2(\Z)$.   Part (5) follows from (4), together with the identity
$$M^T = \begin{pmatrix}0&1\\-1&0\end{pmatrix}M^{-1}\begin{pmatrix}0&-1\\1&0\end{pmatrix}.$$
\end{proof}

All ordered pairs of elements of $\SL_2(R)$ whose first rows coincide forms a bounded family.  This follows from the boundedness of the set of elementary
matrices in $\SL_2(R)$.

An ordered pair $(a,b)$ is \emph{primitive} if and only if the elements $a$ and $b$ generate the unit ideal or, equivalently, if and only if there exists an element of $\SL_2(R)$ whose first row is $\begin{pmatrix} a&b\end{pmatrix}$.  We say a set of ordered pairs $((a,b), (a',b'))$ of primitive pairs is \emph{bounded} if 
the set of pairs $(M,M')$, where $M$ and $M'$ have first rows $\begin{pmatrix} a&b\end{pmatrix}$ and $\begin{pmatrix} a'&b'\end{pmatrix}$ respectively,
is bounded.  We indicate this boundedness condition informally by writing $(a,b)\sim (a',b')$ for pairs in the set.

Given a fixed ring $R$, for polynomials $P_1,Q_1,P_2,Q_2\in \Z[t_1,\ldots,t_k]$ the relation $(P_1,Q_1)\sim (P_2,Q_2)$
we mean the following.  
First, for any $\ba := (a_1,\ldots,a_k)\in R^k$, 
the pair $(P_1(\ba),Q_1(\ba))$ is primitive if and only if the pair $(P_2(\ba),Q_2(\ba))$
is primitive.
Second, the set of pairs $(X_1,X_2)\in \SL_2(R)^2$ such that for some $\ba \in R^k$, the first row of 
$X_i$ is 
$$\begin{pmatrix}  P_i(\ba)& Q_i(\ba)\end{pmatrix},$$
for $i=1,2$,
is bounded.
By Lemma~\ref{equiv}, this makes $\sim$ an equivalence relation on $\Z[x_1,\ldots,x_k]^2$.
\begin{lem}
\label{symbol}
For every ring $R$, we have
$$(t_1,t_2)\sim (t_1,t_2+t_1t_3),$$
and
$$(t_1,t_2)\sim (t_1+t_2t_3,t_2).$$
\end{lem}

\begin{proof}
For $a_1,a_2,a_3\in R$, $(a_1,a_2)$ is primitive if and only if $(a_1,a_2+a_1a_3)$ is primitive, and likewise for $(a_1+a_2a_3,a_2)$.
The boundedness condition follows immediately from the boundedness of the set of elementary matrices in $\SL_2(R)$.
\end{proof}

%
%

%
%

The following argument is due to Vaserstein \cite{V2}.

\begin{prop}
\label{bye-square}
For any ring $R$, we have
$$(1+t_1t_2,t_2^2t_3)\sim (1+t_1t_2,t_3).$$
%

\end{prop}

\begin{proof}
Let $t_1,t_2,t_3$ map to $a,b,c\in R$ respectively.  It is clear that $(1+ab,b^2c)$ primitive implies $(1+ab,c)$ primitive.
Conversely, if $1+ab, b^2c\in J$ for some ideal $J\subsetneq R$, then for any maximal ideal $\fm$ containing $J$, we have $1+ab\in \fm$, hence $b\not\in \fm$,
so $c\in \fm$, which implies $(1+ab,c)$ is not primitive.

For $\mm {x_1}{y_2}{x_3}{y_4},\,\mm{y_1}{x_2}{y_3}{x_4}\in\GL_2(\Q(x_1,\ldots,x_4,y_1,\ldots,y_4))$, setting
$$M:=\mm {x_1}{y_2}{x_3}{y_4}\mm 1101\mm {x_1}{y_2}{x_3}{y_4}^{-1}, N:= \mm{y_1}{x_2}{y_3}{x_4}\mm 10{-1}1\mm{y_1}{x_2}{y_3}{x_4}^{-1},$$
we see that $M$ and $N$ lie in $\SL_2(\Z[x_1,x_3])$ and $\SL_2(\Z[x_2,x_4])$ respectively.  Writing
\begin{equation}
\label{weak-magic}
NMN\mm 0{-1}10 = \mm {P(x_1,x_2,x_3,x_4)}{Q(x_1,x_2,x_3,x_4)}{R(x_1,x_2,x_3,x_4)}{S(x_1,x_2,x_3,x_4)},
\end{equation}
the relation
$$\mm 10{-1}1\mm1101\mm 10{-1}1\mm 0{-1}10 = I$$
implies that
$P-1$, $Q$, $R$, and $S-1$ vanish at $(1,0,0,1)$.  Substituting
$$\mm {x_1}{x_2}{x_3}{x_4} = I + z_5\mm{z_1}{z_2}{z_3}{z_4},$$
we see that $Q$ and $R$ are divisible by $z_5$, so
\begin{equation}
\label{magic}
\mm {z_5}001NMN\mm 0{-1}10\mm {z_5}001^{-1} \in \SL_2(\Z[z_1,\ldots,z_5]).
\end{equation}

Specializing (\ref{weak-magic}) to the case $x_i = y_i = a_i$ for $A=\mm {a_1}{a_2}{a_3}{a_4}\in \SL_2(R)$, we have
\begin{align*}
NMN\mm 0{-1}10 &= A\mm 10{-1}1  \mm 1101   \mm 10{-1}1 A^{-1} \mm 0{-1}10 \\
&= A\mm01{-1}0 A^{-1} \mm 0{-1}10 = A A^T,
\end{align*}
so
$$\biggl\{A A^T\biggm| A\in \SL_2(R)\biggr\}$$
is bounded.  Further specializing to the case $A = I+z_5\mm {z_1}{z_2}{z_3}{z_4}\in \SL_2(R)$, (\ref{magic}) implies the boundedness of
$$\biggl\{\mm{1+z_5z_1}{z_5^2 z_2}{z_3}{1+z_5z_4} \mm{1+z_5z_1}{z_5^2 z_3}{z_2}{1+z_5z_4}
 \biggm| z_1+z_4 + z_5(z_1z_4-z_2z_3)=0\biggr\}$$
The family of pairs 
$$\biggl\{\biggl(\mm{1+z_5z_1}{z_5^2 z_2}{z_3}{1+z_5z_4},  \mm{1+z_5z_1}{z_5^2 z_3}{z_2}{1+z_5z_4}^{-1}\biggr)\biggr\}$$
subject to the condition 
\begin{equation}
\label{sl-cond}
z_1+z_4 + z_5(z_1z_4-z_2z_3)=0
\end{equation}
is therefore bounded, so by part (5) of Lemma~\ref{equiv},  the same is true for the family of pairs 
$$\biggl\{\biggl(\mm{1+z_5z_1}{z_5^2 z_2}{z_3}{1+z_5z_4},  \mm{1+z_5z_1}{z_5^2 z_3}{z_2}{1+z_5z_4}^T\biggr)\biggr\}$$
satisfying (\ref{sl-cond}).
If $(1+ab,b^2c)$ is primitive, then substituting $z_1=a$, $z_2=c$, $z_5=b$, we can solve (\ref{sl-cond}) for $z_3$ and $z_4$ in $R$, which proves the proposition.

\end{proof}

Henceforth, we assume $R$ is the ring of integers in an imaginary quadratic field $K$.

\begin{prop}
\label{bye-b}
For $R$ as above,
$$(1+t_1 t_2, t_2 t_3)\sim (1+t_1 t_2, t_3).$$
\end{prop}

\begin{proof}

By Lemma~\ref{symbol}, for all $d\in R$
$$(1+ab,bc)\sim (1+ab,bc+(1+ab)bd) = (1+ab,b(c+(1+ab)d))$$
and
$$(1+ab,c)\sim (1+ab, c+(1+ab)d)),$$
so we may replace $c$ by any element in the same residue class (mod $1+ab$).  Since $(c,1+ab)$ is primitive, by Hasse's theorem (see \cite[Satz 13, p. 32]{H}) there exist infinitely many choices $d\in R$, such that  $c+(1+ab)d$ generates a prime ideal in $R$.  In particular, replacing $c$ by this element,
we may assume $c$ is relatively prime to $2a$.

We also have for all $e\in R$,
$$(1+ab,bc)\sim (1+ab+bce,bc) = (1+(a+ce)b,bc)$$
and
$$(1+ab,c)\sim (1+(a+ce)b,c).$$
Applying Hasse's theorem again, there exist infinitely many $e\in R$ such that $a+ ce$ is divisible by $4$, and $q := \frac{a+ce}4$ generates a prime ideal of $R$.  
We may therefore  assume $a=4q$ where $q$ generates a prime ideal not dividing $(2)$.
Finally, applying Hasse's theorem a third time, we may choose $p := c + (1+ab)f$ such that
$(p)$ is a prime ideal, and $p\equiv 1\pmod{8q}$. Using the same argument as above, we may replace $c$ by $p$.

For every place $v$ of $K$, let $[-q,p]_v$ denote the Hilbert symbol (which is $1$ if and only if $-qx^2+py^2 = 1$ has a solution in $K_v$ and is $-1$ otherwise).  By Hilbert reciprocity,
$$\prod_v [-q,p]_v = 1.$$
We can restrict the product to finite places of $K$ since the only infinite place is complex.
By Hensel's lemma, if $v$ does not lie over $2$, $-qx^2+py^2 =1$ has a solution in $K_v$ if and only if the reduced equation $-\bar q x^2 + \bar p y^2=1$ has a solution in the residue field $k_v$.  This holds automatically as long as $\bar q$ and $\bar p$ are non-zero in $k_v$, hence over all odd $v$ other than those corresponding to the prime ideals $q$ and $p$.    As $p\equiv 1\pmod 8$,
$\{1-py^2\mid y\in K_v\}$ contains a neighborhood of $0$ if $v$ lies over $2$, and it follows that
$-qx^2+py^2 = 1$ has a solution in $K_v$.  As $p\equiv 1\pmod q$, $-qx^2+py^2=1$ has a solution in the completion of $K$ at $q$.  We conclude that $[-q,p]_v = 1$ when $v$ is the place corresponding to $p$, so the image of $a=4q$ is congruent (mod $p$) to an element of the form $-r^2$, for some $r\in R$. Thus $a \equiv -r^2 \pmod{c}$.

As $ab\equiv -r^2b\pmod{bc}$, 
$$(1+ab,bc)\sim (1-r^2b,bc)\sim (1-r^2b,bc-(1-r^2b)bc) =
(1-r^2b,r^2 b^2 c)$$
Applying Proposition~\ref{bye-square}, this is boundedly equivalent to
$$(1-r^2b,c)\sim (1+ab,c),$$
and the proposition holds.
\end{proof}

For $n$ a non-negative integer and $\alpha = \begin{pmatrix} a&b\\ c&d\end{pmatrix}\in \SL_2(R)$, we write
$$\alpha^n =  \begin{pmatrix} a_n&b_n\\ c_n&d_n\end{pmatrix}.$$
Thus, for each $n$, $a_n$, $b_n$, $c_n$, and $d_n$ can be regarded as polynomials in $a$, $b$, $c$, and $d$ with integer coefficients.

\begin{prop}
\label{power-of-matrix}

For all $n$ and $\alpha$, we have $(a^n,b)\sim (a_n,b_n).$
\end{prop}

\begin{proof}
We define a sequence of polynomials in $t$ as follows:
$$Q_{-1}=-1,\ Q_0 = 0,\ Q_{i+1} = tQ_{i} - Q_{i-1}\;\forall i\ge 0.$$
For $n\ge -1$, we set $u_i := Q_i(\Tr(\alpha))$.
By induction on $i$, we have
$$Q_{i+1}(t)Q_{i-1}(t) = Q_i(t)^2 - 1$$
for $i\ge 0$, so $u_{i}u_{i-2} = (u_{i-1}+1)(u_{i-1}-1)$ 
for all $i\ge 1$.  In particular, we can write $u_n = v_n w_n$, where $v_n$ divides $u_{n-1}-1$ and $w_n$ divides $u_{n-1}+1$.

By the Cayley-Hamilton theorem, $\alpha^2 = \Tr(\alpha)\alpha - I$, so for all $i\ge 1$,
$$\alpha^i = u_i \alpha - u_{i-1}I.$$
In particular,
$$a_n = a u_n - u_{n-1},\ b_n = bu_n = bv_n w_n,$$
so $a_n\equiv 1\pmod{w_n}$ and $a_n\equiv -1\pmod{v_n}$.
By Proposition~\ref{bye-b} and part (4) of Lemma~\ref{equiv},
$$(a_n,b_n)  \sim (a_n, bv_n) \sim (-a_n, -bv_n)\sim (-a_n, -b)\sim (a_n, b).$$
As $\alpha$ is upper triangular $\pmod b$, we have $a_n\equiv a^n\pmod b$, and the proposition follows.
\end{proof}

We recall the following result of Carter and Keller.

\begin{lem}
\label{CK-Lemma4}
(Carter-Keller, see \cite[Lemma 4, p.680]{CK})

Let $F$ be a number field, $\cO$ its ring of integers, and $m$ the number of roots of unity in $F$. Let $\mathfrak{a}$ be a nonzero ideal of $\cO$, and let $b$ be a nonzero element of $\cO$ such that
\begin{itemize}

\item [(i)] $b\cO$ is a prime ideal with residue characteristic prime to $m$, i.e., $b\cO$ is a prime ideal, and $\#(\cO/b\cO)$ and $m$ are relatively prime integers.

\item [(ii)] $\mathfrak{a}$ and $b\cO$ are comaximal, i.e., $\mathfrak{a} + b\cO = \cO$.

\end{itemize}
Then for every unit $u \in \cO^{\times}$, there exists an element $c \in \cO$ such that $bc \equiv u \pmod{\mathfrak{a}}$ and such that the greatest common divisor of $\epsilon(b)$ and $\epsilon(c)$ is $m\gamma$, where $\gamma$ is a positive integer all of whose rational prime divisors ramify in $F/\Q$ (that is, they divide the discriminant of $K$). Furthermore $c$ may be chosen such that $\gamma$ avoids any single rational prime which ramifies in $F/\Q$. Finally if the class number of $K$ is $1$, $c$ may be chosen such that $\gamma = 1$.

\end{lem}

\begin{rmk}
\label{Remark-Lemma4-CK}

The above lemma is Lemma 4 on page 680 in \cite{CK}. In Lemma 4 in \cite{CK}, Carter and Keller made an additional assumption that $\mathfrak{a}$ is a \textit{prime principal} ideal whereas in the above lemma, we do not impose such condition on $\mathfrak{a}$. In fact the proof of Lemma 4 in \cite{CK} given in pages 680--682 does not need such assumption, and thus the proof of the above lemma follows the same lines as that of Lemma 4 in \cite{CK}. Note that in the last paragraph of page 683, Carter and Keller applied Lemma 4 to a nonzero ideal $\mathfrak{a}$ which is not necessarily prime; so it seems that the assumption in Lemma 4 that $\mathfrak{a}$ is prime and principal is a typo in \cite{CK}.

\end{rmk}

\begin{prop}
\label{mth-powers}
Let $m = |R^\times|$.  If $ab\neq 0$ and $(a,b)$ is primitive, then $(a^m,b)\sim (1,0)$.
\end{prop}

\begin{proof}

For $k\in R\setminus \{0\}$, let $\epsilon(k)$ denote the exponent of the finite group $(R/kR)^\times$.

Let $p_1, \ldots, p_{\ell}$ be the distinct prime divisors of the discriminant of $K/\Q$. For each $1 \le i \le \ell$, applying Lemma \ref{CK-Lemma4} for the ideal $a\cO$ and the element $b \in \cO$ with $u = -1$, one obtains an element $c_i \in \cO$ such that the following are satisfied:
\begin{itemize}

\item [(i)] $bc_i \equiv - 1 \pmod{a\cO}$; and

\item [(ii)] the greatest common divisor of $\epsilon(b)$ and $\epsilon(c_i)$ is $m\gamma_i$, where $\gamma_i$ is not divisible by the prime $p_i$ and all prime divisors of $\gamma_i$ divide the discriminant of $K/\Q$.

\end{itemize}

By (ii), note that $\gcd(\gamma_1, \ldots, \gamma_{\ell}) = 1$, and thus there are integers $h_1, \ldots, h_{\ell}$ such that
\begin{align*}
h_1\gamma_1 + \cdots + h_{\ell}\gamma_{\ell} = 1.
\end{align*}

For each $1 \le i \le \ell$, choose $d_i \in \cO$ so that
$$N_i = \begin{pmatrix} a & b \\ c_i & d_i \end{pmatrix}.$$

Choose $x, y \in R$ so that
$$M = \begin{pmatrix} a & b \\ x & y \end{pmatrix} \in \SL_2(R).$$

Then
$$M^m = M^{mh_1\gamma_1}\cdots M^{mh_{\ell}\gamma_{\ell}}.$$

We claim that $M^{mh_i\gamma_i}$ belongs to a bounded subset of $\SL_2(R)$ for all $1 \le i \le \ell$, and thus the same is true of $M^m$. This implies the proposition.

To prove the claim, take an integer $1 \le i \le \ell$. By Proposition \ref{power-of-matrix}, there exist bounded sets $U_1$, $V_1$ in $\SL_2(R)$ such that there exist $X_1 \in U_1$ and $Y_1 \in V_1$ for which 
$$X_1M^{m|h_i\gamma_i|}Y_1 = \begin{pmatrix} a^{m|h_i\gamma_i|} & b \\ x_1 & y_1 \end{pmatrix},$$
for some $x_1, y_1 \in R$. Using the same argument, there exist bounded sets $U_2$, $V_2$ in $\SL_2(R)$ such that there exist $X_2 \in U_2$, $Y_2 \in V_2$ for which
$$ X_2 \begin{pmatrix} a^{m|h_i\gamma_i|} & b \\ x_1 & y_1 \end{pmatrix} Y_2 = N_i^{m|h_i \gamma_i|}.$$

Let $s, t$ be positive integers such that $t - s = m\gamma_i$, $s$ is divisible by $\epsilon(c_i)$, and $t$ is divisible by $\epsilon(b)$. Then
$$(1, 0) \sim (a^t, b) \sim (a_t, b_t),$$
so $N_i^t$ belongs to a bounded subset of $\SL_2(R)$ which does not depend on $(a, b)$. Likewise,
$$(1, 0) \sim (a^s, c) \sim (a_s, c_s),$$
so $(N_i^T)^s$ belongs to a bounded subset of $\SL_2(R)$. By part (5) of Lemma ~\ref{equiv}, $N_i^{-s}$ belongs to a bounded subset of $\SL_2(R)$, so the same is true of $N_i^{m\gamma_i} = N_i^{t - s}$. Thus $N_i^{m|h_i \gamma_i|}$ belongs to a bounded subset of $\SL_2(R)$, so the same is true of $M^{m|h_i\gamma_i|}$. Therefore $M^{mh_i\gamma_i}$ belongs to a bounded subset of $\SL_2(R)$ for all $1 \le i \le \ell$, which proves our claim.

\end{proof}

\begin{thm}
\label{main-thm}

If $R$ is the ring of integers in an imaginary quadratic field, then $\SL_2(R)$ is polynomially parametrized.
\end{thm}

\begin{proof}
It suffices to prove that if $(s,t)\in R^2$ is primitive, then $(1,0)\sim (s,t)$.  By \cite[Lemma 3]{CK} , there exists $a,b\in R$ such that
$(s,t)\sim (a^m,d)$, where $m = |R^\times|$.  Proposition~\ref{mth-powers} implies that $(a^m,d)\sim (1,0)$.
\end{proof}

\begin{rmk}

If $A$ is a commutative ring such that $\SL_2(A)$ is polynomially parametrized, it is natural to ask what \textit{the smallest number of parameters is needed to polynomially parametrize $\SL_2(A)$.} We do not attempt to answer such question in this paper. With a detailed analysis of the proof of Theorem \ref{main-thm}, one can find an \textit{explicit} upper bound for the smallest number of parameters for parametrizing $\SL_2(R)$, where $R$ is the ring of integers of an imaginary quadratic number field. In this direction, there are few results in literature. Vaserstein \cite{V2} proves that an upper bound for the smallest number of parameters for parametrizing $\SL_2(\Z)$ is $46$. Morgan, Rapinchuk, and Sury \cite{MRS} show that if $\cO$ is a ring of $S$-integers in a number field $K$ such that the group of units $\cO^{\times}$ is infinite, then an upper bound for the smallest number of parameters for parametrizing $\SL_2(\cO)$ is $18$. Assuming the truth of Generalized Riemann Hypothesis, Cooke and Weinberger \cite{CW} obtained the same upper bound $18$. If $\cO$ is the ring of integers in a number field, Zannier \cite{Z} proves that a lower bound for the smallest number of parameters for parametrizing $\SL_2(\cO)$ is $4$. It is a natural question as to whether \textit{there exists a uniform bound for the number of parameters needed to parametrize $\SL_2(\cO)$, where $\cO$ is the ring of integers of an arbitrary number field.}

\end{rmk}

\subsection*{Acknowledgements} 

The second-named author thanks Andrei Rapinchuk for useful correspondence regarding Remark \ref{Remark-Lemma4-CK}.

\end{document}